\documentclass[11pt]{article}

\usepackage[usenames,dvipsnames,condensed]{xcolor}

\usepackage{amsthm}
\theoremstyle{definition}
\usepackage{amsfonts}
\usepackage[UKenglish]{babel}
\usepackage[usenames]{xcolor}
\usepackage{graphicx}
\usepackage{soul}
\usepackage{stfloats}
\usepackage{morefloats}
\usepackage{cite}
\usepackage{lscape}
\usepackage{epstopdf}
\usepackage{overpic}

\usepackage{amsmath,amstext,amsopn,amsfonts,eucal,amssymb}
\usepackage{graphicx,wrapfig,url}

\newcommand\Z{{\mathbb Z}}

\newtheorem{theorem}{Theorem}[section]

\newtheorem{lemma}[theorem]{Lemma}
\newtheorem{proposition}[theorem]{Proposition}
\newtheorem{definition}[theorem]{Definition}

\newtheorem{conjecture}[theorem]{Conjecture}

\newcommand{\gr}{{\rm gr}}

\begin{document}

\title{
   Genus  Ranges of   
Chord Diagrams }

\author{Jonathan  Burns,	
Nata\v{s}a Jonoska, 
Masahico Saito
\\
Department of Mathematics and Statistics\\
University of South Florida}

\date{\empty}

\maketitle

\begin{abstract}

A chord diagram consists of a circle, called the backbone, with line segments, called chords, whose endpoints 
are attached to distinct points on the circle. The genus of a chord diagram
is the genus of the orientable surface obtained by thickening the backbone to an annulus and attaching bands  
to the  inner boundary circle at the ends of each chord. Variations of this construction are considered 
here, where bands are  possibly attached to the outer boundary circle of the annulus. The genus range of a 
chord diagram is the genus values over all such variations of surfaces thus obtained from  a given chord 
diagram. Genus ranges  of chord diagrams  for a fixed number of chords are studied. Integer intervals that can, 
and cannot, be realized as genus ranges are investigated. Computer calculations are presented, and play a key 
role in discovering and proving the properties of genus ranges.

\end{abstract}

\section{Introduction}

A chord diagram is a circle (called the backbone) with line segments (called chords) attached at their endpoints. 
Chord diagrams have been extensively used in knot theory and its applications, as well as in physics and biology. 
They are main tools for finite type knot invariants \cite{Dror}, and are also used for describing RNA secondary
structures \cite{APRW}, for example. A chord diagram is usually 
depicted as 
a circle  in the plane with chords inside the circle. The chords may intersect in the circle, but such intersections 
are ignored (chords are regarded as pairwise disjoint). 

The genus of a chord diagram
is the genus of the orientable surface obtained by thickening the backbone to an annulus and attaching bands 
to the  inner boundary circle at the ends of each chord, and it has been studied earlier in
the context of 
knot theory. In \cite{STV}, for example, it was pointed out that the genus of a  chord diagram equals that  of a 
surface obtained from the Seifert algorithm, a standard construction of orientable surfaces bounded by knots 
from diagrams. This fact was used in \cite{STV}
to define the genus of virtual knots  as the minimum of such genera 
over all virtual knot diagrams. Such genera was used in \cite{APRW} for the study of RNA foldings.
Thickened chord diagrams were used for the study of DNA structures as well \cite{Jonoska2002}. 

The genus of a chord diagram is defined by attaching bands at chord endpoints on the inner boundary circle 
of the annulus as mentioned above, and different surfaces could be obtained if some bands are allowed to 
be attached  on the outer boundary circle of the annulus. It is, then, natural to ask which integers arise as 
genera of surfaces  if such variants are allowed for thickened chord diagrams. Specifically, we consider the 
following questions.

\bigskip

\noindent
{\bf Problem.} For a given positive integer $n$, 
(1) determine which sets of integers appear as genus ranges of chord
diagrams with $n$ chords, and 
(2) characterize chord diagrams with $n$ chords that have a specified genus range.

\bigskip

The genus range of graphs 
has 
been studied in topological graph theory~\cite{MoharThomassenBook}. Our focus in this paper is on a special 
class of trivalent graphs that arise as chord diagrams, and 
the 
behavior of their genus ranges for a fixed number of chords.
The genus ranges of 4-regular rigid vertex graphs were studied in \cite{BDJSV}, where the embedding of 
rigid vertex graphs  is 
required to preserve 
the given cyclic order of edges at every vertex.

The paper is organized as follows. Preliminary material is presented in Section~\ref{sec-prelim}. A method of 
computing the genus by the Euler characteristic is given in Section~\ref{sec-comp}, where results of computer 
calculations are also presented. In Section~\ref{sec-prop}, various properties of genus ranges are described, 
and some sets of  integers are realized as genus ranges in Section~\ref{sec-realize}. In Section~\ref{sec-chara},  
 results from Sections ~\ref{sec-prop}
and 
\ref{sec-realize} are combined to summarize our findings on which 
sets of integers can and cannot be realized as  genus ranges of chord diagrams for a fixed number of chords. 
We also list the sets for which realizability as the genus range of a chord diagram has yet to be determined, and end with some short concluding remarks.

\section{Terminology and Preliminaries}\label{sec-prelim}

This section contains the definitions of the concepts, their basic properties, and the 
notations used in this paper. 

A {\it chord diagram} consists of a finite number of {\it chords}, that are closed arcs, with their endpoints 
attached to  a circle, called the {\it backbone}. An example of a chord diagram is given in Fig.~\ref{thicken} (A).
For 
more details  and the background of chord diagrams, see, for example, \cite{Dror,STV}. 

\begin{figure}[htb]
  \begin{center}
    \includegraphics[width=4in]{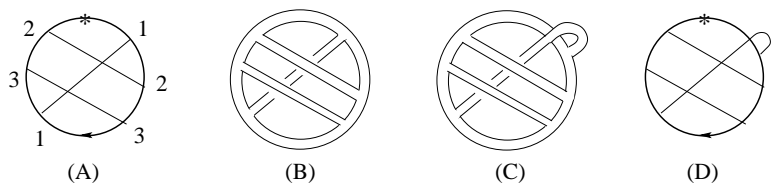}\\
    \caption{(A) An example of a chord diagram 123132,  $*$ indicates the base point;
      (B) and (C) Two examples of thickened chord diagrams corresponding to the chord diagram in (A); 
      (D) Schematic representation of the thickened diagram in (C).}
    \label{thicken}
  \end{center}
\end{figure}

A {\it double-occurrence}  word $w$ over an alphabet set is a word which contains each symbol of the alphabet 
set exactly $0$ or $2$ times. Double-occurrence words are  also called {\it (unsigned) Gauss codes} in knot 
theory \cite{Kauff}.

For a given chord diagram, we obtain a double-occurrence word as follows. If it has $n$ chords, assign distinct 
labels (e.g.,  positive integers $\{ 1, \ldots, n\}$) to the chords. The endpoints of the chords lying on the 
backbone inherit the labels of the corresponding chords. Pick and fix a base point $*$ on the backbone 
of a chord diagram. The sequence of endpoint labels obtained by tracing the backbone in one direction 
(say, clockwise) forms a double-occurrence word corresponding to the chord diagram. Conversely, for a given 
double-occurrence word, a chord diagram corresponding to the word is obtained by choosing distinct points  
on a circle such that each point corresponds to a letter in the word in the order of their appearance, and then 
connecting each pair of points of the same letter by a chord. The chord diagram in Fig.~\ref{thicken} (A) has 
the corresponding double-occurrence word $123132$. Equivalence relations are defined on chord diagrams 
and double-occurrence words in such a way that this correspondence is bijective. Two double-occurrence 
words are equivalent if they are related by cyclic permutations, reversal, and/or symbol renaming.

\smallskip
\noindent 
{\bf Notation.} Applying the above mentioned correspondence between chord diagrams and double-occurrence 
words, in this paper a double-occurrence word $W$ also represents the corresponding chord diagram. 

A {\it thickened} chord diagram (or simply a {\it thickened diagram})  is a compact orientable surface obtained 
from a given chord diagram by thickening its backbone circle and chords as depicted in Fig.~\ref{thicken} (B), (C).
The backbone  is thickened to an annulus. A band corresponding to each chord is attached to one of two 
boundary circles of the annulus. In literature (e.g., \cite{APRW,STV}), all bands are attached to the inner 
boundary of the thickened circle as in Fig.~\ref{thicken} (B), and in this case we say that chords are {\it all-in}, 
or that the chord diagram is of {\it all-in}. For a chord diagram $D$ we denote with $F_D$ the all-in thickened 
chord diagram corresponding to $D$. In this paper, we consider thickened chord diagrams  with band ends 
possibly attached  to the outer boundary circle of the annulus, as is one of the ends of chord $1$ in (C). Since 
each endpoint of a chord has two possibilities of band ends attachments (inner or outer), there are 4 possible 
band attachment cases for each chord, in total $4^n$ surfaces obtained from a chord diagram with $n$ chords. 
To simplify exposition, we draw an endpoint of a chord attached to the outer side of the backbone as in 
Fig.~\ref{thicken} (D) to indicate that the corresponding thickened diagram is obtained by attaching the 
corresponding band end to the outer boundary of the annulus. A band whose one end is connected to the 
outside curve of the annulus and the other is connected to the inside part of the curve is said to be a 
{\it one-in, one-out chord}.

\smallskip
\noindent 
{\bf Convention.} We assume that all surfaces are orientable throughout the paper.

\begin{definition}
Let $g(F)$ denote the genus of a compact  orientable surface $F$. The {\it genus range} of a chord diagram 
$D$ is the set of genera of thickened chord diagrams, and denoted by $\gr(D)$: 
$$ \gr(D)=\{ \, g(F) \ | \  F \mbox{  is a thickened chord diagram of } D \, \} .$$
\end{definition}

\begin{figure}[h]
  \begin{center}
    \includegraphics[width=9cm]{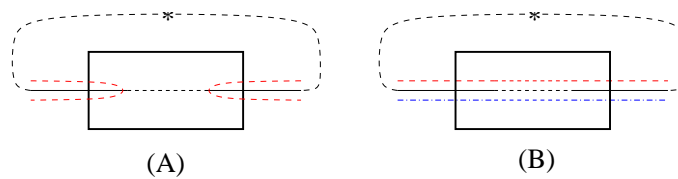} 
    \caption{End edges traced by (A) single and (B) two boundary curve(s).} 
    \label{endedge}
  \end{center}
\end{figure} 

We use the following terminology in the later sections. The closed backbone arc which is a portion of the 
backbone  between  the first and the last endpoints containing the base point is called the {\it end edge}.
Because the backbone and the chords are thickened  to  bands that constitute a thickened diagram, we 
regard that each backbone arc and each chord has two corresponding boundary curve segments, which may 
or may not belong to the same connected component of the boundary. In particular, the boundary curves 
corresponding to the end edge may belong to one or two boundary components, as depicted in 
Fig.~\ref{endedge} (A) and (B), respectively. In each case, we say that the end edge  is {\it traced by a 
single (resp. double) boundary curve(s)}.

\section{Computing the Genus Range of a chord diagram} \label{sec-comp}

In this section we recall the Euler characteristic formula used to compute the genus ranges by counting 
the number of boundary components, and present outputs of computer calculations.

\subsection{Euler characteristic formula}

First we recall the well-known Euler characteristic formula, establishing the relation between the genus and the 
number of boundary components. The Euler characteristic $\chi(F)$ of a compact orientable surface $F$ of 
genus $g(F)$ and the number of boundary components $b(F)$ of $F$ are related by $\chi(F)=2-2 g(F) - b(F)$. 

A thickened chord diagram $F$  is a compact surface with the original chord diagram $D$ as a deformation 
retract. If the number of chords is $n >0$, $n \in \Z$, then there are $2n$ vertices in $D$ and $3n$ edges 
($n$ chords and $2n$ arcs on the backbone), so that $\chi(F)=\chi(D)=2n - 3n=-n$. Thus we obtain the 
following well known formula, which we state as a lemma, as we will use it often in this paper. 

\begin{lemma}     \label{lem-euler}
Let $F$ be a thickened chord diagram of a chord diagram $D$. Let $g(F)$ be the genus of $F$, $b(F)$ be 
the number of boundary components of $F$, and $n$ be the number of chords of $D$. Then we have 
$g(F) =(1/2) (n - b(F) + 2 )$. 
\end{lemma}

Thus we can compute the genus range from the set of  the numbers of  boundary components of thickened 
chord diagrams, $\{ \, b(F) \mid F \mbox{ is a thickened chord diagram of }  D \, \}$. Note that $n$ and 
$b(F)$ have the same parity, as genera are integers.

\subsection{Computer calculations}

In \cite{APRW},  the genera of chord diagrams was defined (which is the genus of all-in chord diagrams), and 
an algorithm to compute the number of graphs with a given genus and $n$ chords by means of cycle 
decompositions of permutations was presented. Our computer calculation is based on a modified version of 
their algorithm. The computational results are posted at \url{http://knot.math.usf.edu/data/} under {\it Tables}.

Computer calculations show that the sets of all possible genus ranges of chord diagrams with $n$ letters for 
$n=1, \ldots, 7$ are as follows.

\begin{table}[h]
\centering

\begin{tabular}{rl}
		& ${\cal GR}_n$\\
$n$ = 1, 2 :	& \{0,1\} \\
$n$ = 3, 4 :	& \{0,1\}, \{0,1,2\}, \{1,2\}  \\
$n$ = 5, 6 : 	& \{0,1\}, \{0,1,2\}, \{1,2\}, \{0,1,2,3\},  \{1,2,3\}  \\
$n$ = 7 : 	& \{0,1\}, \{0,1,2\}, \{1,2\}, \{0,1,2,3\},  \{1,2,3\}, \{0,1,2,3,4\}, \{1,2,3,4\}, \{2,3,4\}
\end{tabular}
\end{table}

The following conjectures hold for all examples we computed.

\begin{conjecture}      \label{conj-range2}
For any $n>0$, if a chord diagram with $n$ chords has genus range consisting of  two numbers, then the 
genus range is either  $\{ 0,1\}$ or $\{1,2\}$.
\end{conjecture}

\begin{conjecture}
For any $n \neq 2$, there is a unique (up to equivalence) double-occurrence word $11 \cdots nn$ that 
corresponds to a chord diagram with the genus range $\{0,1\}$.
\end{conjecture}

We note that there are two 2-letter words,  $1122$ and $1212$, and both corresponding chord diagrams 
have the genus range $\{ 0, 1\}$.

\begin{conjecture}
For any $n \neq 4$, there is a unique   (up to equivalence) chord diagram 
with genus range $\{1, 2\}$, and it is 
$(123123)(44 \cdots nn)$. 
\end{conjecture}

There are several more chord diagrams for $n=4$ with genus range $\{1,2\}$.

\section{Properties of Genus Ranges}\label{sec-prop}

The following is standard for cellular embeddings of general graphs \cite{MoharThomassenBook}, and also 
known for 4-regular rigid vertex graphs \cite{BDJSV}. Below we state the property for chord diagrams. 

\begin{proposition}  \label{lem-consec}
The genus range of any chord diagram  consists of consecutive integers.
\end{proposition}

By Proposition~\ref{lem-consec} the genus ranges of chord diagrams are integer intervals, therefore in 
the rest of the paper we use the notation $[a, b] = \{ k \in \Z \mid a \leq k \leq b\}.$  

\begin{lemma}      \label{lem-nosingleton}
There does not exist a chord diagram whose genus range consists only of a singleton.
\end{lemma}

\begin{proof}
Since all-in thickened diagram $F_D$ for a chord diagram $D$ has an outside boundary component and 
some inside ones, any chord diagram has a thickened chord diagram with more than one boundary component. 
Let $n$ be the number of boundary components of $F_D$, one of which is the outside circle.

Let $c$ be a chord in $D$. Then removing the corresponding band (thickened chord) from $F_D$ either 
increases or decreases the number of boundary components of a thickened diagram by exactly one. 
If  $c$ is traced by a single boundary component, then its band removal splits the component in two parts, and 
if $c$ is traced by two components, then the removal of its band connects the two components 
as 
a single one. Suppose that when a band for $c$  is removed from $F_D$, the number of boundary components 
increases by one. Let $D'$ be the chord diagram with $c$ removed from $D$, and consider $F'=F_{D'}$, the 
all-in thickened diagram of $D'$. Then the number of boundary components of $F'$ is $n+1$. In this case, 
adding the band of $c$ back to $F'$  to obtain $F_D$ will connect two inside boundary components of $F'$. 
Instead, connecting both ends of the band of $c$ to the outside boundary circle of $F'$  increases the number of 
boundary components by one, and gives rise to a thickened diagram of $D$ with $n+2$
boundary components, and with genus $g(F_D)+1$. Hence, $\gr(D)$ is not a singleton. 

We repeat a similar argument for the case when the number of boundary components of $F_D$ decreases by 
one when a band of $c$ is removed. Let $D'$ be the chord diagram with $c$ removed from $D$ and $F'=F_{D"}$ 
be the all-in thickened diagram of $D'$, then the number of boundary components of $F'$ is $n-1$. Adding  to 
$F'$ a one-in, one-out chord for $c$ connects the original inside boundary  with the outside curve and 
decreases the number of boundary components by one. This gives rise to a thickened diagram of $D$ with 
$n-2$ boundary components with genus $g(F_D)-1$.
\end{proof}

\begin{figure}[h]
  \centering
  \includegraphics[width=10cm]{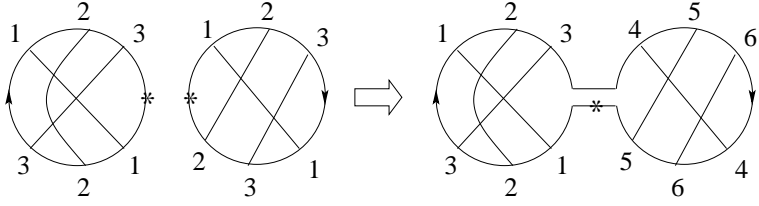}

  \caption{Connected sum of two chord diagrams.}
  \label{sum}
\end{figure} 

The connected sum of two chord diagrams with base points is defined in a manner similar to the connected sum 
of knots, see Fig.~\ref{sum}. A band is attached at the base points preserving orientations to obtain a new 
chord diagram. In the figure, the left and right chord diagrams, respectively, before taking connected sum are 
represented by double-occurrence words $W_1=123123$ and $W_2=123132$, respectively, and after the 
connected sum, it is represented by $W=123123456465$, after renaming $W_2$. We use the notation 
$W=W_1 W_2$ to represent the word thus obtained, by renaming and concatenation. 

\begin{lemma}\label{lem-connectsum}
Let $W_1$ and $W_2$ be chord diagrams such that the genus ranges of corresponding chord diagrams are 
$[g_1, g_1']$ and $[g_2, g_2']$, respectively. Let $e_1,e_2$ be the end edges of $W_1$, $W_2$, respectively. 
Then the genus range of the chord diagram corresponding to $W=W_1 W_2$ is 
$[g_1 + g_2 - \epsilon, g_1' + g_2' - \epsilon ']$
for some $\epsilon, \epsilon ' \in \{ 0, 1\}$, where $\epsilon, \epsilon'$ are determined as follows.

\begin{description}
\item [($E_0$)] $\epsilon=0$ if and only if at least one of the end edges (say, the end edge $e_1$  of $W_1$) has the 
following property:  any thickened graph of genus $g_1$ traces $e_1$ by two boundary curves.

\item[($E_1$)] $\epsilon=1$ if and only if  both end edges  $e_1$ and $e_2$ have the following property:  there exist 
thickened graphs of genus $g_1$ and $g_2$, respectively, that trace both $e_1$ and $e_2$ by a single boundary 
curve.

\item[($E'_0$)] $\epsilon '=0$  if and only if at least  one of the end edges (say, $e_1$ of $W_1$) has the following 
property:  there exists a thickened graph of genus $g_1'$ that traces $e_1$ by two boundary curves.

\item[($E'_1$)] $\epsilon '=1$ if and only if both end edges $e_1$ and $e_2$ of $W_1$ and $W_2$ have the 
following property:  any thickened graphs of genus $g_1'$ and $g_2'$, respectively, trace $e$ and $e'$ by a single  
boundary curve.
\end{description}

\begin{figure}[h]
  \centering
  \includegraphics[width=15cm]{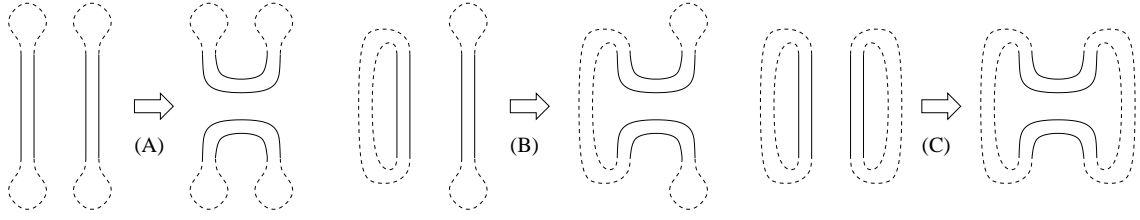}

  \caption{Connected sum and boundary curves: (A) both end edges are traced by a single boundary curve; 
    (B) one end edge is traced by a single boundary curve but the other is traced by two boundary curves; 
    (C) both end edges are traced by two boundary curves.} 
  \label{connect}
\end{figure} 

\end{lemma}

\begin{proof}
This is proved by a case-by-case analysis of the number of boundary components and by using 
Lemma~\ref{lem-euler}. A similar argument is  found in  \cite{BDJSV,Jonoska2002}. Let $n_1, n_2$ be the 
number of chords of  chord diagrams corresponding to $W_1$, $W_2$, respectively. The number of chords 
of $W=W_1 W_2$ is $n=n_1+ n_2$. Let $b_1$ and $b_2$ be the number of boundary components of 
thickened chord diagrams for $W_1$ and $W_2$, respectively. The number of boundary component $b$ of a 
thickened diagram $D$ of $W=W_1 W_2$ after taking the connected sum equals $b_1+b_2-\alpha$ where 
$\alpha=0$ or $\alpha =2$. If both end edges $e_1$ and $e_2$   are traced by a single component (the 
situation as in Fig.~\ref{connect} (A)) then $\alpha=0$. If at least one end edge $e_1$ or $e_2$ is traced 
by two components (the situations Fig~\ref{connect} (B) and (C)), then $\alpha = 2$. Then 
Lemma~\ref{lem-euler} implies
\begin{align*}    
g 	&= \frac{1}{2} \left( n - b + 2 \right) = \frac{1}{2} \left[ (n_1 + n_2) - (b_1 + b_2 - \alpha) + 2 \right] \\
 	&= \frac{1}{2} \left(n_1 - b_1 + 2 \right) + \frac{1}{2} \left( n_2 - b_2 + 2 \right) + \frac{\alpha}{2} - 1 \\
 	&= g_1 + g_2  + \frac{\alpha}{2} - 1, 
\end{align*}
where $g_1$, $g_2$, and $g$ are genera of $W_1$, $W_2$, and $W$, respectively. 

For statement  $(E_1)$, there are thickened diagrams with minimal genus of $W_1$ and $W_2$ for which 
$\alpha = 0$ and whose connected sum preserves the number of boundary components (Fig.~\ref{connect} (A)), 
hence the statement follows. The other cases are proved by similar arguments. 
\end{proof}
 
Since we often refer to the number of boundary components tracing the end edge, we define the following 
notation. Let $e$ be the end edge of a chord diagram corresponding to a double-occurrence word $W$.
Let $c$ be 1 or 2. We say that $W$ satisfies the condition $A({\rm min}, c)$ (resp. $A({\rm max}, c)$) 
if any thickened diagram of minimum (resp. maximum) genus $g$  traces $e$ by a single boundary curve for 
$c=1$, and by two boundary curves for $c=2$. Similarly, we say that $W$ satisfies the condition 
$E({\rm min}, c)$ (resp. $E({\rm max}, c)$) if there exists a thickened diagram of minimum (resp. maximum) 
genus $g$ that  traces $e$ by a single boundary curve for $c=1$, and by two boundary curves for $c=2$. 
We also simply say $W$ is (of) $A({\rm min}, c)$ etc. Then Lemma~\ref{lem-connectsum} is summarized as follows.

\begin{table}[h]
\centering

  \begin{tabular}{clc} 
  \hline
    Cases & \quad $W_1$,  $W_2$ & $W_1 W_2$ \\
  \hline
    $(E_0)$ 	& one $A({\rm min}, 2)$ 	& $ \epsilon = 0 $  \\
    $(E_1)$ 	& both $E({\rm min}, 1)$	& $ \epsilon = 1 $  \\
    $(E_0')$ 	& one $E({\rm max}, 2)$ 	& $ \epsilon ' = 0 $ \\
    $(E_1')$ 	& both $A({\rm max}, 1)$ 	& $ \epsilon ' = 1 $ \\
  \hline
  \end{tabular}
\end{table}

If a chord diagram $D'$ is obtained from $D$ by removing some chords, then $D'$ is called a {\it sub-chord diagram} 
of $D$. The following lemma covers a large family of chord diagrams that support Conjecture~\ref{conj-range2}.

\begin{lemma}      \label{lem-range3words}
If  a  chord diagram  has a sub-chord diagram corresponding to  the double-occurrence word $123321$, then 
its genus range contains more than $2$ integers.
\end{lemma}

\begin{proof}
Consider the surface $F$ obtained by thickening the chord diagram such that the three parallel chords represented 
by 1, 2 and 3 are all-in, and the other chords are all-out. Then it has $4$ inside boundary curves and at least one 
outside, total at least $5$. Refer to Fig.~\ref{parallel3}, where other chords are not depicted. Move one end of  
chord $1$ from inside to outside, keeping the other inside. Then the total number of boundary curves decreases 
by $2$. This is seen as follows. Regard this operation in two steps: (1) remove a band corresponding to chord 
$1$ from $F$, and (2) add a band corresponding to $1$ with one-in and one-out ends. The step (1) joins the two 
inside boundary curves to a single curve, thus reduces the number of boundary curves by 1. In step (2), the new 
one-in, one-out band joins the newly formed inside curve with one of the outside curves, reducing the boundary 
curve by 1 again. Hence  replacing an all-in chord $1$ in $F$ with one-in, one-out chord decreases the number of 
boundary curves by 2. Performing the same procedure for the chord labeled $3$, further decreases the number 
of components by $2$. Therefore, the genus range consists of at least $3$ numbers.
\end{proof}

\begin{figure}[h]
  \centering
  \includegraphics[width=7cm]{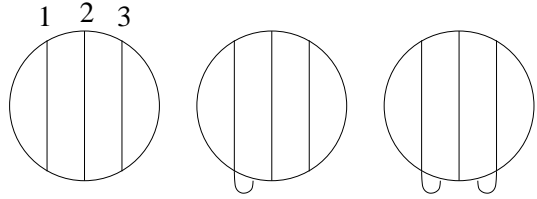} 

  \caption{Three thickened surfaces with distinct genera for a chord diagram containing the sub-chord diagram 
    $123321$. All other chords (not pictured) are all-out.}  
  \label{parallel3}
\end{figure}

\section{Realizations of Genus Ranges}\label{sec-realize}
 
We use the following notations for respective double-occurrence words and corresponding chord diagrams: 
\begin{align*}      
U_n 	& = 1122\cdots nn \quad (\; =\ {U_1}^n\; ) , \\
R_n 	& = 12\cdots n 12 \cdots n , \\
G_\ell &= (1212)(3434) \cdots ((2\ell -1)2 \ell (2\ell -1) 2 \ell) \quad (\; =\ {R_2}^{\ell}\; ) .
\end{align*}

\begin{lemma}      \label{lem-repeatword}
For   the chord diagram corresponding to $R_n=12\cdots n 12 \cdots n$, where $n=2m$ or $n=2m-1$ and 
$n>2$, we have  $\gr(R_n)=[1, m]$. 
\end{lemma}

\begin{proof}
For an even $n=2m$ ($m\ge 1$), consider the all-in thickened diagram $F_{R_n}$. Then $F_{R_n}$ has 
exactly two boundary components:  One inside curve, tracing chords in successive order 
(see Fig.~\ref{repeat} (A)), and one outside. Hence $F_{R_n}$ achieves the maximum genus $m$.  By 
adding a one-in, one-out chord, the two curves are joined to a single component, therefore for an odd $n$, 
the resulting surface the maximum genus. 

Consider a thickened diagram for $R_n$ where every chord is one-in, one-out (see Fig.~\ref{repeat} (B)). 
Each boundary curve traces a single side of two chords. Then the resulting surface has $n$ boundary 
components, and genus $1$.

Since the chord diagram $D$ for $R_3=123123$ is isomorphic (as a graph) to the bipartite graph $K_{3,3}$, 
$D$ is non-planar and the genus range of any chord diagram that has $R_3$ as a sub-chord diagram does 
not contain $0$. The result follows from Lemma~\ref{lem-consec}.
\end{proof}

\begin{figure}[h]
  \centering
  \begin{overpic}[width=7cm]{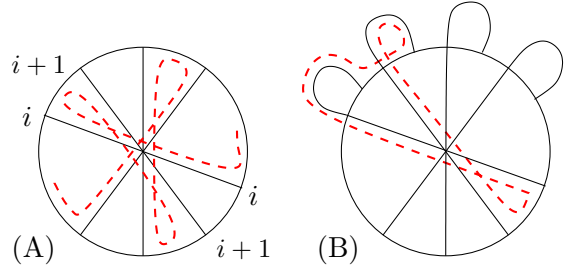} 
    \put(-5,0){(A)}
    \put (40,10){$i$}  
    \put (-3,25){$i$}
    \put (34,0){{\small $i+1$}}  
    \put (-5,35){{\small $i+1$}}
    \put(53,0){(B)}
  \end{overpic}

  \caption{Thickened diagrams of $R_n$ with the maximum genus (A) and the minimum genus (B).} 
  \label{repeat}
\end{figure} 

\begin{lemma}      \label{lem-range2}
(1)  For any $n>0$, there exists a chord diagram of $n$ chords with genus range $[0,1]$.
(2)  For any $n>2$, there exists a chord diagram of $n$ chords with genus range $[1,2]$.
\end{lemma}

\begin{proof}
The chord diagram $U_1=11$ has genus range $[0,1]$  and also has the properties $A({\rm min}, 2)$ and 
$A({\rm max}, 1)$. By Lemma~\ref{lem-connectsum} (cases $(E_0)$  and $(E'_1)$), the chord diagram of 
$U_2={U_1}^2$ has genus range $[0,1]$, and its end edge retains the conditions $A({\rm min}, 2)$ and 
$A({\rm max}, 1)$. Inductively, $U_n$ has genus range $[0,1]$ for any $n$.

Recall that the chord diagram corresponding to $R_3$ is non-planar. The chord diagram of $R_3$ has genus 
range $[1,2]$ (Lemma~\ref{lem-repeatword}), and has property $A({\rm max}, 1)$. Then the chord diagram 
of $R_3 U_1$ has genus range $[1,2]$ by Lemma~\ref{lem-connectsum} (cases $(E_0)$ and $(E'_1)$),  
and retains the condition $A({\rm max}, 1)$. Inductively, $R_3 U_m$ has genus range $[1,2]$ for any 
$m\geq 0$, hence for any $n=m+3>2$. 
\end{proof}

\begin{lemma}\label{lem-1212}
For any chord diagram $W$ with $\gr(W)=[g, g']$, we have $\gr(R_2 W)=[g, g'+1]$. 
\end{lemma}

\begin{proof}
The chord diagram  $R_2$ has genus range $[0,1]$ and is of $A({\rm min}, 2)$ and $A({\rm max}, 2)$, 
so it is $E({\rm max}, 2)$. By  Lemma~\ref{lem-connectsum}  (cases $(E_0)$ and $(E'_0)$), we obtain 
the result.
\end{proof}

\begin{lemma}\label{lem-repeatrepeat}
For  $G_m=(1212)(3434)\cdots ( (2m-1)2m(2m-1)2m )$, we have  $\gr(G_m)= [0, m]$ for any $m>0$. 
\end{lemma}

\begin{proof}
This follows from Lemma~\ref{lem-1212} by induction.
\end{proof}

\begin{lemma}\label{lem-repeatrepeat11}
 For any $k, m >0 $, we have  $\gr(G_m U_k)=[0, m+1]$.
\end{lemma}

\begin{proof}
The chord diagram  $G_m$ is of $E({\rm max}, 2)$. By Lemma~\ref{lem-connectsum} $(E_0')$,  
we have $\gr(G_n U_1 )=[0, m+1]$. The chord diagrams $U_1$ and  $G_m U_k$ for $k>0$ 
are of $A({\rm max}, 1)$, hence by Lemma~\ref{lem-connectsum} $(E'_1)$  the statement holds by 
induction.
\end{proof}

We use the notation $X=12341342$.

\begin{lemma}\label{lem-highgenus}
For any $k>0$,  $k \in \Z$,  there exists a chord diagram with $n$ chords, where $n=4k-1$ or $4k$,
having genus range $[k, 2k]$.
\end{lemma}

\begin{proof}
Computer calculation shows that the chord diagram $D$ corresponding to $W=12312345674675$ has 
genus range $[2,4]$. The word $W$  is the concatenation of $R_3=123123$ and $X=12341342$. 
Computer calculation also shows that $\gr(X)=[1,2]$. By Lemma~\ref{lem-repeatword}, we also have 
$\gr(R_3)=[1,2]$.

\begin{figure}[h]
  \centering
  \includegraphics[width=3cm]{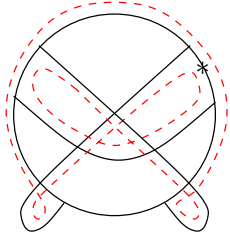} 
    \put(-13,60){$*$} 
  \caption{Edges traced by a single boundary curve of a genus $1$ surface.}
  \label{F123123}
\end{figure}

The diagram of $R_3$  is of $E({\rm min}, 1)$ as depicted in Fig.~\ref{F123123}. This implies that 
$X=12341342$ is of $A({\rm min}, 2)$. (Otherwise $R_3 X$ has minimum genus 1 by 
Lemma~\ref{lem-connectsum} $(E_0)$.) A connected sum of two chord diagrams of $A({\rm min}, 2)$ 
is again a diagram of $A({\rm min}, 2)$ (case (C) in Fig.~\ref{connect}). By Lemma~\ref{lem-connectsum} 
$(E_0)$  again, inductively, the minimum genus of $X^k$ is $k$, and $X^k$ is of $A({\rm min}, 2)$.

Any thickened diagram of $R_3$ with genus $2$ must have a single boundary component, 
and therefore, every edge is singly traced, hence it is $A({\rm max}, 1)$. Since $\gr(W)=[2,4]$,
Lemma~\ref{lem-connectsum} $(E'_0)$ implies that $X$ is of $E({\rm max}, 2)$. Note that the end edge 
$e$ of $X^k$ for any $k$ is of $E({\rm max}, 2)$ (case (C) in Fig. ~\ref{connect}). By using 
Lemma~\ref{lem-connectsum} $(E'_0)$ inductively, we obtain that the maximum genus of $X^k$ is $2k$. 

Hence the diagram for $X^{k}$ has genus range $[k, 2k]$ and $4k$ chords. The diagram corresponding 
to $X^{k-1} R_3  $ has $n=4k-1$ chords, the minimum genus $(k-1)+1=k$, the maximum genus 
$2(k-1)+2=2k$, by Lemma~\ref{lem-connectsum}  as desired.
\end{proof}

We note here that computer calculation was critical for this proof, since it would otherwise be difficult to 
determine the genus range of $R_3 X=12312345674675$.
 
The proof of Lemma~\ref{lem-highgenus} shows that $X^k$ is of $A({\rm min}, 2)$ and $E({\rm max}, 2)$ 
for every $k>0$.

\begin{lemma}\label{lem-hk}
For any $h>0$ and $ k\geq 0$, there is a chord diagram  $W(h, k)$ of $n=4k+h$ chords with genus range 
$[k, 2k+1]$.
\end{lemma}

\begin{proof}
Let $W(h, k)= U_h X^{k}$ which has $n=4k+h $ chords. The diagram $U_1$ is of $A({\rm min}, 2)$, and 
inductively, so is $U_h$ for any $h>0$. Lemma~\ref{lem-connectsum} $(E_0)$  implies that  $W(h, k)$ has 
the minimum genus $k$.

We have that $U_h$ has genus range $[0,1]$ and is of $A({\rm max}, 1)$, and $X^k$ is of $E({\rm max}, 2)$ 
(Proof of Lemma~\ref{lem-highgenus}). Hence $W(h,k)$  has highest genus $2k+1$. The statement holds 
for $k=0$ as well, from the proof of Lemma~\ref{lem-range2} (1).
\end{proof}

\begin{lemma}\label{lem-morehk}
For any $h>0$ and $k, \ell  \geq 0$, there is a chord diagram $V(h,  k, \ell)$ with $n= 4k + 2\ell +h $ chords 
such that  $\gr(V (h,k, \ell) )=[k, 2k + \ell + 1 ]$. In the case $h=0$, for any $k, \ell  \geq 0$, there is a 
chord diagram $V(0,  k, \ell)$ with $n= 4k + 2\ell $ chords such that  $\gr(V (0,k, \ell) )=[k, 2k + \ell  ]$. 
\end{lemma}

\begin{proof}
We consider $V(h, k, \ell)= W(h,k) G_\ell = U_h X^k G_\ell$. We see that $G_\ell$ is of $A({\rm min}, 2)$ 
and $E({\rm max}, 2)$ inductively from the proof of Lemma~\ref{lem-1212}. Hence the  minimum genus of 
$V (h,k,\ell )$ is $k$ for any $h, k, \ell \geq 0$ by Lemma~\ref{lem-connectsum} $(E_0)$. 

Recall that $X^k$ is of $E({\rm max}, 2)$ and $G_\ell$ is of $E({\rm max}, 2)$ for any $k, \ell >0$.   
 Hence $V(0, k, \ell)$ is of $E({\rm max}, 2)$ and $\gr( V(0, k, \ell)) = \gr(X^kG_\ell) =[k, 2k+ \ell]$ by 
Lemma~\ref{lem-connectsum} $(E'_0)$, proving the second statement of the lemma. Because $V(0, k, \ell)$ 
is of $E({\rm max}, 2)$ and $\gr( U_h)=[0,1]$, Lemma~\ref{lem-connectsum} $(E_0')$ implies  
$V (h,k, \ell) = [k, 2k + \ell + 1 ]$ for $h>0$, $k, \ell \geq 0$.
\end{proof}

\begin{proposition}
For any $g, g'$ such that $g' \ge 2g $ there is a chord diagram with genus range $[g, g']$.
\end{proposition}

\begin{proof}
For $g' >2g$, we set $\ell=g' - 2g$. Then the chord diagram $V(0,k,\ell)$ in 
Lemma~\ref{lem-morehk} has genus range $[g, g']$. If $g'=2g > 0$  then Lemma~\ref{lem-highgenus} 
provides a desired chord diagram.
\end{proof}

\section{Towards Characterizing Genus Ranges} \label{sec-chara}

In this section we state and prove the main theorem. Recall from Lemma~\ref{lem-euler} that any chord 
diagram of $n$ chords, the genus $g$ of a thickened diagram is at most $\lceil n/2 \rceil$. 

\begin{theorem}
There exists  a chord diagram with $n$ chords and genus range $[g, g']$ whenever $g, g'$  satisfy one 
of the following conditions: 
(1) $g'=2g$ and either $g=1$ or $g'=\lceil n/2 \rceil$, or
(2) $0 \leq  2 g < g' \leq \lceil n/2 \rceil$.
\end{theorem}

\begin{proof}
Let $m=\lceil n/2 \rceil $.
Case (1):  The case of $[g,g']=[1,2]$  follows from Lemma~\ref{lem-range2} (2). In the case of  
$0< g'=2g=m$, setting $k=g$ in Lemma~\ref{lem-highgenus}, we obtain a chord diagram with genus 
range $[k, 2k]=[g, g']$ with $n=2m=4k$ or $n=2m-1=4k-1$ as desired.

Case (2): Suppose  $0 \leq  2 g < g' \leq m$. First  we consider the case $n=2m-1$. Set 
$\ell=g'-2g-1 \geq 0$ and $h=2m- 2g' +1 >0$. Then $V(h, g, \ell)$ in Lemma~\ref{lem-morehk}
has genus range $[g, 2g+\ell +1]=[g, g']$ and the number of chords is 
$4g + 2 \ell  + h= 4g + 2( g' - 2g  -1) + (2m-2g' +1 )= 2m-1=n, $ as desired. 

Next we consider the case $n=2m$. If $m=g'$,  set $\ell= g' - 2g >0$. Then the  chord diagram 
$V(0, g,  \ell)$ in Lemma~\ref{lem-morehk} has genus range $[g,  2g + \ell]=[g, g']$ and the number 
of chords $4g + 2 \ell = 2g'=2m=n$, as desired. If  $m> g'$, set  $\ell = g' - 2g -1$ and 
$h'=  m - g' = m - 2g - \ell \geq 0$. Then the chord diagram corresponding to $V(2h', g, \ell)$ in 
Lemma~\ref{lem-morehk} has genus range $[g,  2g + \ell + 1 ]=[g, g']$ and the number of chords 
$4g + 2 \ell + 2h' = 2g' + 2h'=2m=n$, as desired. 
\end{proof}

\begin{figure}[h]
  \centering
  \begin{overpic}[width=6cm]{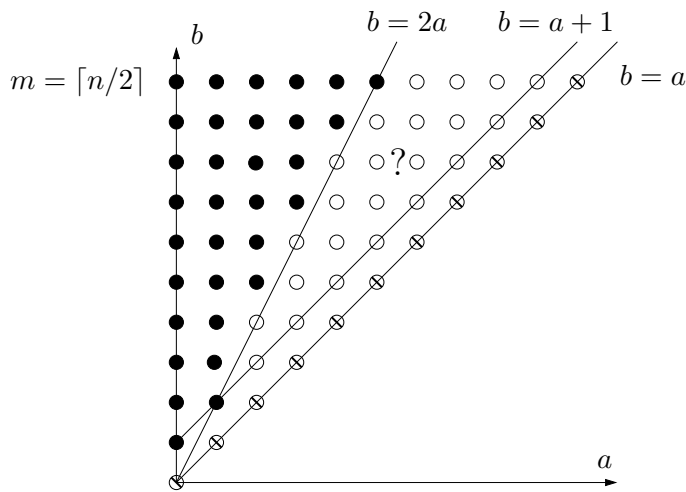}
    \put(-35,89){$m=\lceil n/2 \rceil $}
    \put(5,99){$b$}
    \put(44,102){$b=2a$}
    \put(73,102){$b=a+1$}
    \put(100,90){$b=a$}
    \put(49,70){\Large ?}
    \put(95,5){$a$}
  \end{overpic}
  \caption{Realizing genus ranges for chord diagrams with $n$ chords.}
  \label{gr}
\end{figure}

 The situation of the theorem is represented in the graph of Fig.~\ref{gr}. Each lattice point of coordinate 
$(a,b)$ represents the genus range $[a,b]$. A black dot represents that there is a chord diagram of the 
corresponding genus range. A circle with backslash inside, located on the line $b=a$, represents that there 
is no singleton genus range by Lemma~\ref{lem-nosingleton}. White dots between two lines $b=a$ and 
$b=2a$, and those on the line $b=2a$, denote the cases for which we do not know whether there are 
diagrams of those ranges. Note that only two points are realized on the lines $b=2a$ and $b=a+1$. Other 
points on the integer lattice, not indicated in the figure, are excluded from the Euler characteristic formula 
(Lemma~\ref{lem-euler}).

\section{Concluding Remarks}      \label{sec-concl}

In this paper, we studied  sets of  genus values, called the genus ranges, for thickened chord diagrams. 
Variations of surfaces occur when bands that correspond to chords are attached to outside circle boundary 
of the backbone of a chord diagram. Computer calculations and constructive methods were used to prove 
the results. For a fixed number of chords, we investigated which ranges can and cannot occur. It may be 
of interest to investigate the ranges for which we have not been able to determine  whether they can be 
realized or not.

\section*{Acknowledgements}
This research was partially supported by National Science Foundation  DMS-0900671 and National Institutes 
of Health  R01GM109459. The content is solely the responsibility of the authors and does not necessarily 
represent the official views of the NSF or NIH.


\end{document}